\documentclass{amsart}
\usepackage{amssymb}

\usepackage[T1]{fontenc}
\usepackage{mathpazo,courier}
\usepackage[scaled]{helvet}
\usepackage{graphicx}

\usepackage{subfigure}
\usepackage[usenames,dvipsnames]{xcolor}
\definecolor{darkblue}{rgb}{0.0, 0.0, 0.55}
\usepackage[pagebackref,colorlinks,linkcolor=BrickRed,citecolor=OliveGreen,urlcolor=darkblue,hypertexnames=true]{hyperref}
\usepackage{enumerate}
\newtheorem{theorem}{Theorem}[section]
\newtheorem{lemma}[theorem]{Lemma}
\newtheorem{corollary}[theorem]{Corollary}

\newtheorem{proposition}[theorem]{Proposition}
\theoremstyle{definition}
\newtheorem{definition}[theorem]{Definition}
\newtheorem{example}[theorem]{Example}

\theoremstyle{remark}
\newtheorem{remark}[theorem]{Remark}

\numberwithin{equation}{section}
\usepackage{ushort}
\newcommand\X{\ushort X}
\newcommand\U{\ushort U}

\newcommand{\Real}{{\mathbb R}}
\newcommand{\C}{{\mathbb C}}

\newcommand{\Q}{\mathbb Q}
\newcommand{\R}{{\mathbb  R}}
\newcommand{\N}{{\mathbb N}}

\newcommand{\x}{\mathbf{x}}
\newcommand{\uu}{\mathbf{u}}
\newcommand{\y}{\mathbf{y}}
\newcommand{\z}{\mathbf{z}}

\newcommand {\ZZ} {{\rm Zer}}

\newcommand {\hide}[1]{}
\usepackage[applemac]{inputenc}
\begin{document}
\title[Irreducible components]
{On irreducible components of real exponential hypersurfaces}
\author{Cordian Riener}
\address{Aalto Science Institute, Aalto University, P.O. Box 13000, 00076 Aalto, Finland}
\address{Department of Mathematics and Statistics, Faculty of Science and Technology, University of Tromsø, 9037 Tromsø,Norway}
\email{cordian.riener@aalto.fi}
\author{Nicolai Vorobjov}
\address{Department of Computer Science, University of Bath, Bath
BA2 7AY, England, UK}
\email{nnv@cs.bath.ac.uk}

\begin{abstract}
Fix any real algebraic extension $\mathbb K$ of the field $\Q$ of rationals.
Polynomials with coefficients from $\mathbb K$ in $n$ variables and in $n$ exponential functions
are called {\em exponential polynomials over $\mathbb K$}.
We study zero sets in $\Real^n$ of exponential polynomials over $\mathbb K$, which we call {\em exponential-algebraic sets}.
Complements of all exponential-algebraic sets in $\Real^n$ form a Zariski-type topology on $\Real^n$.
Let $P \in {\mathbb K}[X_1, \ldots ,X_n,U_1, \ldots ,U_n]$ be a polynomial and denote
$$V:=\{ (x_1, \ldots , x_n) \in \Real^n|\> P(x_1, \ldots ,x_n,, e^{x_1}, \ldots ,e^{x_n})=0 \}.$$
The main result of this paper states that, if the real zero set of a polynomial $P$ is irreducible over $\mathbb K$
and the exponential-algebraic set $V$ has codimension 1, then, under Schanuel's conjecture over the reals,
either $V$ is irreducible (with respect to the Zariski topology) or each of its irreducible components  of codimension 1
is a rational hyperplane through the origin.
The family of all possible hyperplanes is determined by monomials of $P$.
In  the case of a single exponential (i.e., when $P$ is independent of $U_2, \ldots , U_n$) stronger statements are
shown which are independent of Schanuel's conjecture.
\end{abstract}

\maketitle
\section{Introduction}

The main motivation of this paper is to begin a study of irreducible components of zero sets of functions
defined by expressions that are polynomial in variables and exponentials of variables.
The components are meant to be defined by the same type of expressions.
The interesting questions include: do the components have any special structure and what is an upper bound on their number.

Throughout this article we will  denote tuples of variables by
$\X:=(X_1,\dots,X_n)$ and $\U:=(U_1,\ldots,U_n)$, and with a given tuple $\X$ we associate the tuple of exponential functions
$e^{\X} := (e^{X_1}, \ldots ,e^{X_n})$.
We consider the field of real algebraic numbers $\Real_{alg}$, the field $\Q$ of rational numbers, and we fix
a real algebraic extension $\mathbb K$ of $\Q$.
Further,  ${\mathbb K}[\X,\U]:={\mathbb K}[X_1,\dots,X_n,U_1,\ldots,U_n]$ will denote the ring of polynomials with coefficients
in $\mathbb K$ in the $2n$  variables.
Clearly, ${\mathbb K}[\X,e^{\X}]$ is a ring of functions which we call the ring of \emph{exponential polynomials over $\mathbb K$}
(or \emph{$E$-polynomials}, for brevity).
The geometry and model theory of zero sets of $E$-polynomials are well understood
(see for example \cite{complexity, few, decidability}).
In the special case when $P$ is independent of variables $\X$, the $E$-polynomial is called {\em exponential sum}
(with integer spectrum).
The theory of zero sets of exponential sums is also developed, in particular in \cite{Kazar97, Zilber02},
but apparently not for the structure of their irreducible components.

Every $P\in {\mathbb K}[\X,\U]$ defines an $E$-polynomial $f$ via the map
$$E:\> {\mathbb K}[\X,\U] \to {\mathbb K}[\X,e^{\X}],$$
such that $f(\X)=E(P(\X, \U))=P(\X,e^{\X})$.

A finite set of polynomials $\mathcal{P}:=\{P_1,\ldots,P_k\}\subset {\mathbb K}[\X,\U]$ defines a real algebraic set
$$\ZZ(\mathcal{P}):=\{ (\x,\uu)=(x_1, \ldots ,x_n,u_1,\ldots,u_n) \in \Real^{2n}|\> P_1(\x,\uu)=\cdots=P_k(\x,\uu)=0  \},$$
and similarly, we will denote by $\ZZ(f_1,\ldots,f_l)\subset\R^n$ the zero set of a finite set of $E$-polynomials.
We will call any $\ZZ(f_1,\ldots,f_l)\subset\R^n$ a {\em (real) exponential-algebraic set} or just {\em exponential set},
for brevity.
By taking the sum of squares, every real algebraic set (respectively, every exponential set) can be defined as a zero set
of a single polynomial (respectively, $E$-polynomial).
An exponential set $V$ will be called reducible (over $\mathbb K$) if there are two distinct non-empty exponential sets
$V_1,\ V_2$ such that $V= V_1 \cup V_2$, and irreducible otherwise.
It will be shown in Section~\ref{sec:prelim} that every exponential set can be uniquely represented as a finite union
of irreducible exponential subsets (called irreducible components) neither of which is contained in another.

In this article we will be concerned with the structure of irreducible components and will prove the following main result.

\begin{theorem}\label{thm:main}
Let $P\in {\mathbb K}[\X,\U]$ and assume that $\ZZ(P)\subset\Real^{2n}$ is an irreducible real algebraic set.
Further, let $f=E(P)$ and assume that the dimension (see Definition~\ref{def:dimension} below) of $\ZZ(f)$ is $n-1$.
Then, assuming Schanuel's conjecture, either $\ZZ(f)$ is also irreducible, or every of its $(n-1)$-dimensional
irreducible components is a rational hyperplane through the origin.
\end{theorem}

Schanuel's conjecture is formulated at the beginning of Section~\ref{sec:main_theorem}.
In the case of a single exponential term, i.e., when $P$ is independent of all, but possibly one,
variables $U_1, \ldots ,U_n$, the theorem can be made stronger and independent of Schanuel's conjecture
(see Theorem~\ref{thm:oneexpgeneral} below).

Let us illustrate Theorem~\ref{thm:main} by some examples.

\begin{example}\label{ex:irred}
Let $P= 2X-U+1$.
The straight line $\ZZ(P) \subset \Real^2$ is, of course, irreducible.
According to Theorem~\ref{thm:main}, the exponential set $V= \ZZ(2X-e^X+1)$ (consisting of two points) is irreducible.
This can also be seen by the following elementary argument, independent of Schanuel's conjecture,
which in various modifications is used throughout this paper.
Suppose $V$ is reducible.
The only way to split two-point set into two distinct non-empty parts is the partition into singletons.
Then the non-zero point $A \in V$ can be defined as $A= \ZZ(Q(X,e^X))$ for some $Q(X,U) \in {\mathbb K}[X,U]$.
It follows that $A$ is the projection along $U$ of an isolated point in $\ZZ(P,Q)$ with algebraic coordinates.
This contradicts the Lindemann theorem.
\end{example}

\begin{example}
Let $P:= X_1U_2+X_2U_1-X_1-X_2$ and $f:=E(P)$.
Then $\dim \ZZ(P)=3$ and $\dim \ZZ(f)=1$.
The polynomial $P$ is irreducible over ${\mathbb K}$ (even over $\Real$), hence the algebraic set $\ZZ(P)$
is irreducible over $\mathbb K$.
On the other hand, $\ZZ(f) \subset \Real^2$ is reducible and consists of two irreducible components, which are the lines
$\ZZ(X_1)$ and $\ZZ(X_2)$.
\end{example}

\begin{figure}
    \subfigure[The transformed Cartan umbrella
     ]{\includegraphics[width=0.49\textwidth]{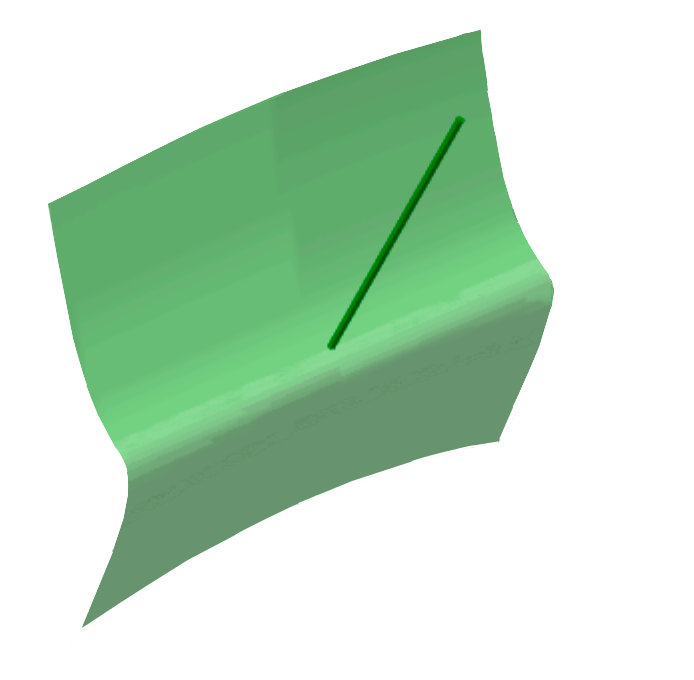}}
    \subfigure[The intersection with $e^{X}=U$ in the line $L$ and the point $A$]
    {\includegraphics[width=0.49\textwidth]{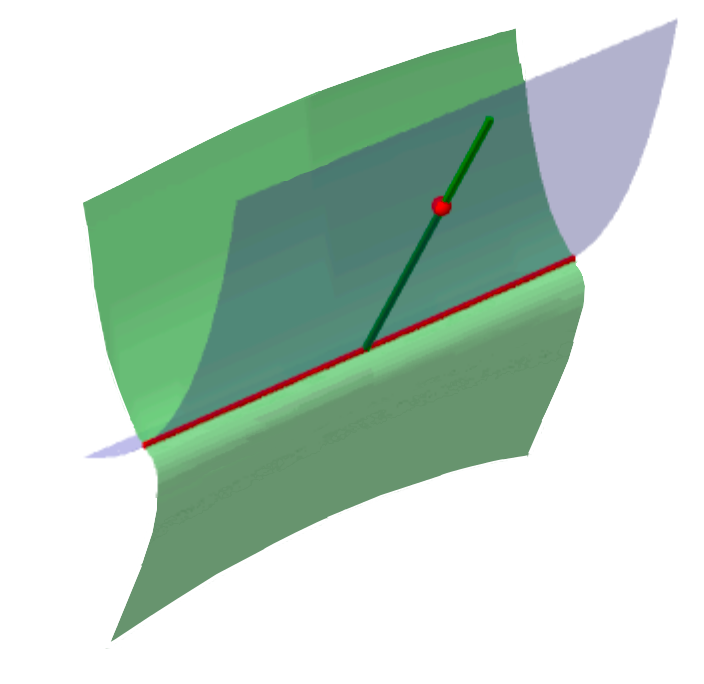}}
\caption{Visualization of Example~\ref{ex:intro}}
\end{figure}

\begin{example}\label{ex:intro}
This example illustrates the case when for an irreducible $\ZZ(P)$ the exponential hypersurface $\ZZ(E(P))$
has irreducible components of codimension greater than 1.

Consider the polynomial
$$P:=(X_1+U-1)((2X_1-U+1)^2+X_2^2)+(2X_1-U+1)^3.$$
Note that $\ZZ(P)$ is an affine transformation the {\em Cartan umbrella} \cite{BCR}.
The algebraic set $\ZZ(P) \subset \Real^3$ contains the straight line $L:= \ZZ( U-1,X_1)$, therefore,
it intersects with the surface $\ZZ(U-e^{X_1})$ along this line.
It also contains the straight line $\ZZ(2X_1-U+1, X_2 \}$, which intersects with $\ZZ(U-e^{X_1})$ by exactly two points,
$(0,0,1) \in L$ and another point, $A$, with transcendental coordinates.
We now prove that
$$\ZZ(P) \cap \ZZ(U-e^{X_1}) = L \cup \{ A \}.$$

Let $f(X_1,X_2):=P(X_1,X_2, e^{X_1})$.
Note that $X_1+e^{X_1}-1$ equals $0$ if and only if $X_1=0$, hence for $X_1 \neq 0$ the equation $f=0$ can be rewritten as
\begin{equation}\label{eq:eq}
X_2^2=-(2X_1-e^{X_1} +1)^2 \left( \frac{2X_1-e^{X_1}+1}{X_1+e^{X_1}-1}+1 \right).
\end{equation}
Now we prove that
\begin{equation}\label{eq:ineq}
\frac{2X_1-e^X_1+1}{X_1+e^X_1-1} > -1
\end{equation}
for all $X_1 \neq 0$.
If $X_1>0$ then $X_1+e^{X_1}-1>0$, hence (\ref{eq:ineq}) is equivalent to
$$2{X_1}-e^{X_1}+1>-{X_1}-e^{X_1}+1,$$
i.e., to $3{X_1}>0$, which is obviously true.
If ${X_1}<0$, then ${X_1}+e^{X_1}-1<0$, so (\ref{eq:ineq}) is equivalent to
$$2{X_1}-e^{X_1}+1<-{X_1}-e^{X_1}+1,$$
or $3{X_1}<0$, which again is true.
It follows that (\ref{eq:eq}) can hold true (for ${X_1} \neq 0$) if and only if ${X_2}=2{X_1}-e^{X_1}+1=0$, and our claim is proved.

The polynomial $P$ is irreducible over $\Real$, hence the algebraic set $\ZZ(P) \subset \Real^3$ is irreducible over $\mathbb K$.
On the other hand, the set $\ZZ(f) \subset \Real^2$ is reducible over $\mathbb K$, with two irreducible components: one-dimensional
$\ZZ(X_1)$ and zero-dimensional
$$\ZZ( 2{X_1}-e^{X_1}+1, X_2)$$
which consists of two points, rational $(0,0)$ and transcendental projection of $A$ along $U$.
\end{example}

\section{Regular and singular points}\label{sec:prelim}

The aim of this section is to prove that taking Zariski closure of a set of all points of a given local dimension
does not increase the dimension (Theorem~\ref{th:subset}).
The idea is to consider the singular locus of the ambient exponential set.
To that end, we adjust to the exponential case the standard routine of definitions and statements
regarding regular and singular points.
We note some differences with the algebraic case, in particular that the dimension of a singular locus of an
exponential set may remain the same as the dimension of the set.

Recall that in the introduction we defined the map
$$E:\> {\mathbb K}[\X,\U] \to {\mathbb K}[\X,e^{\X}],$$
such that $E(P(\X, \U))=P(\X,e^{\X})$, for every $P \in {\mathbb K}[\X,\U]$.

\begin{lemma}\label{le:isomorphism}
The map $E$ is an isomorphism of rings.
\end{lemma}
\begin{proof}
It is immediate that $E$ is an epimorphism of rings.
Thus, it remains to establish that $E$ is injective.
In order to argue by contradiction, assume that $E$ is not injective.
Thus, there exist distinct $P,Q \in {\mathbb K}[\X,\U]$ such that the functions $E(P)$ and $E(Q)$ coincide.
Since $P$ and $Q$ are distinct, the algebraic set $\ZZ(P-Q)$ has  dimension at most $2n-1$.
Also, $\ZZ( \U-e^{\X} ) \subset \ZZ(P-Q)$.
Observe that the dimension of the Zariski closure $Z$ of
$\{ \x \in \Real^n |\> \{ \x \} \times \Real^n \subset  \ZZ(P-Q)\}$ is less than $n$.
Take a point $\x=(x_1, \ldots ,x_n) \in \Real_{alg}^n \setminus Z$ such that its coordinates are linearly independent over $\Q$.
We have that $(\x, e^{\x}) \in \ZZ(\U-e^{\X}) \subset \ZZ(P-Q)$, hence the  numbers
$x_1, \ldots ,x_n, e^{x_1}, \ldots ,e^{x_n}$ are algebraically dependent over $\Q$.
But now it follows from the Lindemann-Weierstrass theorem that numbers $x_1, \ldots x_n$
are linearly dependent over $\Q$ which contradicts the choice of $\x$.
\end{proof}

Lemma~\ref{le:isomorphism} immediately implies the following corollary.

\begin{corollary}
The ring ${\mathbb K}[\X, e^{\X}]$ is Noetherian.
\end{corollary}

Given an exponential set $V\subset\Real^n$, let $I(V)\subset {\mathbb K}[\X,e^{\X}]$ denote the set
of all $E$-polynomials in ${\mathbb K}[\X,e^{\X}]$ that vanish on $V$.
It is easy to see that $I(V)$ is an ideal in ${\mathbb K}[\X,e^{\X}]$.

The following corollary is an immediate implication of Noetherianity.

\begin{corollary}
The ideal $I(V)$ is finitely generated.
\end{corollary}

The next corollary is a standard implication of Noetherianity (see \cite[Proposition~1.1]{hartshorne}).

\begin{corollary}\label{le:topology}
The complements of all exponential sets in $\Real^n$ form a Noethrian topology on $\Real^n$.
We call this topology the {\em Zariski topology}.
\end{corollary}

\begin{definition}
A non-empty subset $\mathcal Y$ of a topological space $\mathcal X$ is called {\em irreducible} if it cannot be
represented as a union ${\mathcal Y}= {\mathcal Y}_1 \cup {\mathcal Y}_2$ of its proper subsets ${\mathcal Y}_1$ and
${\mathcal Y}_2$ each of which is closed in $\mathcal Y$.
For an empty set $\mathcal Y$ irreducibility is undefined.
\end{definition}

Applying this definition to ${\mathcal X}= \Real^n$, equipped with the Zariski topology defined in Corollary~\ref{le:topology},
and $\mathcal Y$ being an exponential set, we get the definition of an irreducible exponential set.

The following corollary is another standard implication of Noetherianity (see \cite[Proposition~1.5]{hartshorne}).
\begin{corollary}
Every non-empty exponential set $V \subset \Real^n$ can be represented as a finite union $V=V_1 \cup \cdots \cup V_t$ of
irreducible exponential sets $V_i$.
If we require $V_i \not\subset V_j$ for $i \neq j$, then the sets $V_i$ are defined uniquely and are called {\em irreducible components}
of $V$.
\end{corollary}

We borrow the definition of a regular point of an exponential set from real algebraic geometry (see \cite[Definition~3.2.2]{BR}).
Let $V \subset \Real^n$ be an exponential set and $I(V)=(f_1, \ldots ,f_k)$ its ideal generated by exponential
polynomials $f_1, \ldots ,f_k \in {\mathbb K}[\X,e^{\X}]$.
Set $r= \sup_{\x \in V} {\rm rank}\ (V, \x)$, where
$$
{\rm rank}\ (V, \x) := {\rm rank}\ \left( \frac{\partial (f_1, \ldots ,f_k)}{\partial (X_1, \ldots ,X_n)}(\x) \right).
$$
Note (see \cite{BR}) that the number $r$ does not depend on the choice of the set of generators $f_1, \ldots ,f_k$.
The number $r$ is called the {\em rank} of the ideal $I(V)$, and we write $r= {\rm rank}\ I(V)$.

\begin{definition}
Let $V \subset \Real^n$ be an exponential set.
\begin{enumerate}
\item If $V$ is irreducible then a point $\x \in V$ is called a {\em regular} point of $V$ if
${\rm rank}\ (V, \x) = {\rm rank}\ I(V)$.
\item  If $V$ is reducible  then $\x \in V$ is called {\em regular} point of $V$ if
there exists one and only one irreducible component of $V$ containing $\x$.
\item A point $\x \in V$ that is not  regular it is called a {\em singular} point of $V$.
\item  We denote the subset of all regular points of $V$ by ${\rm Reg}\ (V)$ and the subset of all singular points
of $V$ by ${\rm Sing}\ (V)$.
\item An exponential set is called {\em non-singular} if $V= {\rm Reg}\ (V)$.
\end{enumerate}
\end{definition}

The proof of the following lemma is exactly the same as the proof of Proposition~3.2.4 in \cite{BR} for algebraic sets.

\begin{lemma}\label{le:contained}
The set ${\rm Sing}\ (V)$ is an exponential set properly contained in $V$.
\end{lemma}

\begin{lemma}\label{le:singfilt}
Let $V \subset \Real^n$ be an exponential set.
Then there is a finite filtration
\begin{equation}\label{eq:singfilt}
V \supset {\rm Sing}\ (V) \supset {\rm Sing}\ ({\rm Sing}\ (V)) \supset
{\rm Sing}\ ({\rm Sing}\ ({\rm Sing}\ (V))) \supset \cdots
\end{equation}
such that the last set in the chain is non-singular.
\end{lemma}

\begin{proof}
Termination of the chain follows from Noetherianity since for every exponential set $V \subset \Real^n$ the set ${\rm Sing(V)}$
is again an exponential set.
The last set in this chain is non-singular because otherwise the set of its singular points would define a proper subset.
\end{proof}

\begin{remark}\label{re:partition}
It follows from Lemma~\ref{le:singfilt} that we can define a ``non-singular stratification'' of an exponential set
$V \subset \Real^n$ as the following finite partition
$$
V= {\rm Reg }\ (V) \cup {\rm Reg }\ ({\rm Sing }\ (V)) \cup {\rm Reg}\ (({\rm Sing}\ ({\rm Sing}\ (V))))
\cup \cdots.$$
\end{remark}

The proof of the following lemma is exactly the same as the proof of Proposition~3.2.9 in \cite{BR} for algebraic sets.

\begin{lemma}\label{le:3.2.9}
If $V\subset\Real^n$ is an irreducible exponential set, then for every $\x \in {\rm Reg}\ (V)$ there exists a Zariski neighbourhood
$U$ of $\x$ in $\Real^n$, and exponential polynomials $f_1, \ldots ,f_r \in I(V)$ such that
$V \cap U = U \cap \{f_1= \cdots =f_r=0 \}$, and
$$
{\rm rank}\ \left( \frac{\partial (f_1, \ldots ,f_r)}{\partial (X_1, \ldots ,X_n)} (\y) \right)
= {\rm rank}\ I(V)=r
$$
for every $\y \in V \cap U$.

In particular, ${\rm Reg}\ (V)$ is a real analytic submanifold of $\Real^n$ of dimension $n-r$.
\end{lemma}

Remark~\ref{re:partition} and Lemma~\ref{le:3.2.9} allow to define the notion of the {\em dimension} of an exponential set.

\begin{definition}\label{def:dimension}
Let $V \subset \Real^n$ be an exponential set.
\begin{enumerate}
\item The {\em dimension} of  $V $, denoted by $\dim V$, is the maximal dimension of
real analytic manifolds comprising the non-singular stratification of $V$.
The dimension $\dim ({\rm Reg}\ (V))$ of the set of regular points of $V$ is its dimension as a real analytic manifold.
\item If $\dim V=n-1$, we call $V$ a {\em hypersurface} in $\Real^n$.
\item For  $\x \in V$ we denote by  $\dim_{\x} V$ the {\em local dimension} of $V$ at $\x$, i.e.,
the maximal dimension of real analytic manifolds
which are intersections of elements of the non-singular stratification of $V$ with an Euclidean neighbourhood of $\x$ in $\Real^n$.
\end{enumerate}
\end{definition}

For an exponential set $V \subset \Real^n$ one can also introduce the analogy of the Krull dimension as follows.

\begin{definition}
The {\em Krull dimension} of $V$, denoted by $\dim_K (V)$, is  largest $d\in\N$ such that
there is a filtration $V_0 \subset V_1 \subset \cdots \subset V_d$ of pair-wise distinct irreducible exponential subsets of $V$.
\end{definition}

Unlike the case of real or complex {\em algebraic} sets, $\dim (V)$ does not necessarily coincide with
$\dim_K (V)$ as is shown in the following example.

\begin{example}
Consider the irreducible two-point exponential set
$$V:=\ZZ(2X-e^X+1) \subset \Real^1$$
from Example~\ref{ex:irred}.
Observe that $T:= \ZZ(X)$ is also irreducible.
Since we have $T\subsetneq V$ it follows that $\dim_K V=1$, while $\dim V=0$.
Further, we can conclude from   $\dim_K V=1$ that  $\dim_K \Real^1 \ge 2$.
On the other hand $\dim_K \Real^1 \le 2$, since  every nonempty irreducible exponential subset in $\Real^1$,
consisting of a finite number of points, having a point different from $T$, and properly containing another such set,
would contain an algebraic point different from $T$, which contradicts  Lindemann's theorem.
So we conclude that  $\dim_K \Real^1 = 2$.
Note that every exponential set in $\Real^1$ which is irreducible over $\Real$ (rather than over $\mathbb K$)
is an irreducible {\em algebraic} set over $\Real$ (actually, a singleton).
Hence, the Krull dimension of $\Real^1$ is 1 in this case.
\end{example}

In view of this example we emphasize that in the sequel we will be using the concept of  dimension
exclusively in the sense of Definition~\ref{def:dimension}.

We can also observe that unlike the case of real or complex {\em algebraic} sets,
the dimension of an exponential set $V$ may coincide with the dimension of its singular locus ${\rm Sing}\ (V)$
as shown in the following example.

\begin{example}
Let $V:= \ZZ\left((2X-e^X+1)(3X-e^X+1)\right) \subset \Real^1$.
Clearly, $ T:=\ZZ(X) \subset {\rm Sing}\ (V)$ (in fact, $ T = {\rm Sing}\ (V)$) but
$\dim (T)= \dim (V)$.
\end{example}

\begin{lemma}\label{le:ineq}
For any exponential set $V$ one has $\dim ({\rm Sing}\ (V)) \le \dim ({\rm Reg}\ (V))$.
\end{lemma}

\begin{proof}
Assume first that $V$ is irreducible.
We prove the statement by induction on the length of the filtration (\ref{eq:singfilt}).
If (\ref{eq:singfilt}) consists of a single exponential set, $V$, then this set is non-singular, so the base of induction is true.

Since, by Lemma~\ref{le:contained}, $I({\rm Sing}\ (V)) \supset I(V)$, we conclude that
$${\rm rank}\ I({\rm Sing}\ (V)) \ge {\rm rank}\ I(V).$$
By Lemma~\ref{le:3.2.9},
$$\dim ({\rm Reg}\ (V))= n- {\rm rank}(I(V))$$
and
$$\dim ({\rm Reg}\ ({\rm Sing}\ (V)))= n- {\rm rank}(I({\rm Sing}\ (V))_),$$
hence
\begin{equation}\label{eq:one}
\dim ({\rm Reg}\ (V)) \ge \dim ({\rm Reg}\ ({\rm Sing}\ (V))).
\end{equation}

By the inductive hypothesis,
$$
\dim ({\rm Reg}\ ({\rm Sing}\ (V))) \ge \dim ({\rm Sing}\ ({\rm Sing}\ (V))),
$$
which, together with the previous inequality, implies that
\begin{equation}\label{eq:two}
\dim ({\rm Reg}\ (V)) \ge \dim ({\rm Sing}\ ({\rm Sing}\ (V))).
\end{equation}

Since
$$
\dim ({\rm Sing}\ (V))= \max \{ \dim ({\rm Reg}\ ({\rm Sing}\ (V))),\ \dim ({\rm Sing}\ ({\rm Sing}\ (V))) \},
$$
we conclude from (\ref{eq:one}) and (\ref{eq:two}) that
$\dim ({\rm Reg}\ (V)) \ge \dim ({\rm Sing}\ (V))$.

Now suppose that $V$ is reducible and $V^{(1)},\ V^{(2)}$ are two of its irreducible components.
Let $V^{(3)}= V^{(1)} \cap V^{(2)}$.
Since $\dim \left( V^{(3)} \right) \le \min \{ \dim \left( V^{(1)} \right),\ \dim \left( V^{(2)} \right) \}$
and, by the first half of the proof,
$$\dim \left( V^{(1)} \right)= \dim \left( {\rm Reg}\ \left( V^{(1)} \right) \right),\> \dim \left( V^{(2)} \right)
= \dim \left( {\rm Reg}\ \left( V^{(2)} \right) \right),$$
we get
$\dim \left( V^{(3)} \right) \le \dim ({\rm Reg}\ (V))$.
\end{proof}

\begin{corollary}\label{cor:local}
Let $\dim (V) =r$ and $\dim_{\x} (V) <r$ for some $\x \in V$.
Then $\x \in {\rm Sing}\ (V)$.
\end{corollary}

\begin{proof}
Let, contrary to the claim, $\x \in {\rm Reg}\ (V)$.
By Lemma~\ref{le:3.2.9}, $\dim_{\x} (V)$ is the same at every point in ${\rm Reg}\ (V)$.
Since, by Lemma~\ref{le:ineq}, $\dim {\rm Reg}\ (V) = \dim (V)$, we conclude that $\dim_{\x} (V)=r$
which is a contradiction.
\end{proof}

\begin{definition}\label{def:set_of _local_dim}
For an exponential set $V$ and $0 \le p \le \dim (V)$ denote $V_p:= \{ \x \in V|\> \dim_{\x} (V)=p \}$.
\end{definition}

\begin{theorem}\label{th:subset}
Let $V_p \neq \emptyset$ for some $0 \le p \le \dim (V)$.
There is an exponential subset $W \subset V$ such that $\dim (W)=p$ and $V_\ell \subset W$ for all $0 \le \ell \le p$.
\end{theorem}

\begin{proof}
We prove that $W$ can be found among the sets of the filtration defined in  (\ref{eq:singfilt}).
If $p= \dim (V)$ we take $W=V$.
Let $p < \dim (V)$.
By Corollary~\ref{cor:local},  $V_p \subset {\rm Sing}\ (V)$.
Passing from an exponential set $S_i$ in the chain (\ref{eq:singfilt}) to the next one on the right, $S_{i+1}$,
consists of removing from $S_i$ an equidimensional subset ${\rm Reg}\ (S_i)$ having the highest dimension, $\dim (S_i)$.
On the other hand, (\ref{eq:singfilt}) cannot consist only of sets of dimension strictly greater than $p$, since
the last set in (\ref{eq:singfilt}) is a non-singular equidimensional set, while $V_p \neq \emptyset$.
Hence, there is the first (from the left) set in the filtration having the dimension $p$, which can be taken as $W$.
\end{proof}

\begin{remark}\label{re:less_exp}
For an integer $k$ such that $0 \le k \le n$, let
$$\X_k:= (X_1, \ldots , X_k),\> \U_k:= (U_1, \ldots ,U_k),\> \text{and}\> e^{\X_k} := (e^{X_1}, \ldots ,e^{X_k}).$$
All definitions in this section can be extended to the ring of functions ${\mathbb K}[\X, e^{\X_k}]$ (polynomials in
$n$ variables $X_1, \ldots ,X_n$ and $k$ exponentials $e^{X_1}, \ldots ,e^{X_k}$).
In particular, we consider Zariski topology in $\Real^n$ with all closed sets of the kind $\ZZ (f)$, where
$f \in {\mathbb K}[\X, e^{\X_k}]$.
If $\ZZ (f)$ is irreducible with respect to this topology, we will say that it is
{\em irreducible in} ${\mathbb K}[\X, e^{\X_k}]$.
It is easy to check that all statements in this section, including Theorem~\ref{th:subset}, hold true for an arbitrary
fixed $k$.
\end{remark}

\section{Case of a single exponential}\label{sec:one_exp}

In this section we consider the case of exponential sets that involve only a single exponential,
i.e., exponential sets defined by $E$-polynomials
$f=P(X_1, \ldots , X_n, e^{X_1})$ with $P \in {\mathbb K}[\X, U_1]$.
We denote  $V:=\ZZ(f)$ and  $m:=\dim(V)$.
Let $\pi: \> \Real^{n+1} \to \Real^{n}$ be the projection map along $U_1$.

\begin{definition}
A real algebraic set $W \subset \Real^{n+1}$ is called {\em admissible} for $V$ if
$$V= \pi \left( \left( W \cap  \ZZ \left( U_1-e^{X_1} \right) \right) \cup \ZZ(P, X_1, U_1-1) \right)$$
and $\dim (W) \le m+1$.
\end{definition}

\begin{lemma}
There exists an admissible set for $V$.
\end{lemma}

\begin{proof}
A proof of the statement immediately follows from \cite[Section~7]{RV94}.

Alternatively, if the set
\begin{equation}\label{eq:index}
\{ \ell \in {\mathbb Z}|\> 0 \le \ell \le m+1,\> (\ZZ (P) \setminus \ZZ(X_1))_\ell \neq \emptyset \}
\end{equation}
is non-empty, then let  $r$ denote its maximal element.
By Theorem~\ref{th:subset} (or its easier version for algebraic sets), there is an
$r$-dimensional algebraic set $W \subset\ZZ(P)$ containing the semialgebraic set $(\ZZ(P))_r$
as well as all sets $(\ZZ(P))_\ell$ where $0 \le \ell \le r$.
If the set (\ref{eq:index}) is empty, assume $W=\emptyset$.
Then $W$ is an admissible set for $V$, according to Theorem~\ref{th:transverse} in the Appendix.
\end{proof}

\begin{lemma}\label{le:oneexp}
Let $W^{(1)}, \ldots ,W^{(t)}$ be all irreducible components of an admissible set $W$ for $V$.
Let $A$ be an $m$-dimensional irreducible component of $V$.
Then, either $A$ coincides with $\pi (W^{(i)} \cap \ZZ(U_1-e^{X_1}))$ for some $1\leq i\leq t$, or
$A$ is the union of an $m$-dimensional algebraic subset of $ \ZZ(X_1)$ and a (possibly empty) set of points having local
dimensions less than $m$.
\end{lemma}

\begin{proof}
Since $A$ is irreducible, it is either a subset of $\ZZ (X_1)$ or a subset
of the projection $\pi((W^{(i)} \cap \ZZ( U_1-e^{X_1}))$ for a certain $1\leq i\leq t$.
In the first case the proof is completed.
So assume that $A \subset \pi (W^{(i)} \cap \ZZ(U_1-e^{X_1}))$.
If $A = \pi (W^{(i)} \cap \ZZ(U_1-e^{X_1}))$, then the proof is completed.
Suppose now that $A \neq \pi (W^{(i)} \cap \ZZ(U_1-e^{X_1}))$.
Then there exists another irreducible component, $B$, of $V$ such that $B \subset \pi (W^{(i)} \cap \ZZ(U_1-e^{X_1}))$.
Let $W^{(i)}= \ZZ(R) $ and $A= \pi ( \ZZ( Q, U_1-e^{X_1}))$ for some polynomials $R, Q \in {\mathbb K}[\X,U_1]$.
Thus, there are two algebraic sets, $W^{(i)}=\ZZ(R)$ and $\ZZ(Q)$ with a non-empty intersection, and $\ZZ(R)$ is irreducible.
Since we have $B \not\subset \pi (\ZZ(Q))$ it follows that  $\ZZ(R) \not\subset \ZZ( Q)$.
Then, $\dim (\ZZ(R, Q)) < \dim (\ZZ( R))$, hence $\dim (\ZZ( R, Q))=m$.
Therefore, the set $S:= \ZZ(Q, U_1 - e^{X_1})$ is an $m$-dimensional real analytic subset of $m$-dimensional
algebraic set $\ZZ(R, Q)$.

By Lindemann's theorem, the set $S \setminus \ZZ(U_1-1, X_1)$ does not contain points with algebraic coordinates.
Hence, $\dim_{\x} S <m$ at every $\x \in S \setminus \ZZ(U_1-1, X_1)$.
It follows that $A$ is the union of an algebraic set $\pi (\ZZ(Q, U_1-1, X_1))$ and a set of points having local
dimensions less than $m$.
\end{proof}

\begin{corollary}\label{re:codim1}
Let $\ZZ(P)$ be an irreducible algebraic set and let $\dim (V)=m=n-1$.
Then $V$ is either irreducible or every $(n-1)$-dimensional irreducible component of $V$ is the union of the hyperplane
$\ZZ(X_1)$ and a set of points having local dimensions less than $n-1$.
\end{corollary}

\begin{proof}
We can choose the set $\ZZ(P)$ as an admissible set for $V$, and then apply Lemma~\ref{le:oneexp}.
\end{proof}

\begin{theorem}\label{thm:oneexpgeneral}
Every $m$-dimensional irreducible component of $V$ either coincides with $\pi (W^{(i)} \cap \ZZ(U_1-e^{X_1}))$
for some $1\leq i\leq t$, or is an algebraic subset of $ \ZZ(X_1)$.
\end{theorem}

\begin{proof}
It follows from  Lemma~\ref{le:oneexp} that if an irreducible component $A$ of $V$ does not coincide with one of
$\pi (W^{(i)} \cap \ZZ(U_1-e^{X_1}))$,  it is the union of a $m$-dimensional algebraic subset of $ \ZZ(X_1)$
and a set $B$ of points having local dimensions not exceeding $p<m$.
Let $C:= B \setminus \ZZ(X_1)$, then $C \subset V_p$.
To argue by contradiction suppose that $C \neq \emptyset$.
According to Theorem~\ref{th:subset}, there is an exponential subset $T \subset V$ such that $\dim (T)=p$ and $V_p \subset T$.
Hence, $A$ can be represented as the union of two distinct non-empty exponential sets, $A \cap \ZZ(X_1)$ and $A \cap T$.
It follows that $A$ is reducible, which is a contradiction.
\end{proof}

\begin{example}
Let $P:= X_1^2+(X_2^2+ (U_1-1)^2-1)^2$.
Note that $\ZZ(P) \subset \Real^3$ is a 1-dimensional set (a unit circle, centered at $(0,0,1)$) in the coordinate
plane $\ZZ(X_1)$, hence is irreducible.
Let $V=\ZZ(E(P))$.
We can choose the admissible family for $V$ consisting of the unique set $\ZZ(P)$.
Then $V$ is the projection of the 0-dimensional algebraic set $\ZZ(P, U_1-1, X_1)$,
is reducible, and consists of two irreducible components, $\ZZ(X_1^2 + (X_2 + 1)^2)$ and $\ZZ( X_1^2 + (X_2 - 1)^2)$.
\end{example}

The following corollary is immediate.

\begin{corollary}
If $\ZZ(P)$ is an irreducible algebraic set and $\dim (V)=n-1$, then $V$ is either irreducible or it has a unique
$(n-1)$-dimensional irreducible component coinciding with the hyperplane $\ZZ(X_1)$.
\end{corollary}

For  the case of a reducible $V$ this corollary can be illustrated by Example~\ref{ex:intro}.

\begin{corollary}\label{cor:bezout}
The number of all irreducible components of $V$ does not exceed $(cd)^n$, where
$c$ is an absolute positive constant and $d$ is the total degree of polynomial $P$.
\end{corollary}

\begin{proof}
Theorem~1 in \cite{RV02} implies that the sum of numbers of all (absolutely) irreducible components of Zariski closures
of sets $(\ZZ(P))_\ell$ over all $\ell=0, \ldots, m$ does not exceed $(c_1d)^n$ for an absolute positive constant $c_1$.
Also the number of all irreducible components of the algebraic set
$\ZZ(P, U_1-1, X_1)$ does not exceed $(c_2d)^n$ for an absolute positive constant $c_2$.
Now the corollary follows from Theorem~\ref{thm:oneexpgeneral}.
\end{proof}

Observe that the bound in Corollary~\ref{cor:bezout} is asymptotically tight because it is tight already for polynomials.

\section{Case of many exponentials}

Consider  a polynomial $P \in {\mathbb K}[\X, \U]$.
Then every monomial of $P$, with respect to the variables $U_1, \ldots , U_n$, is of the kind
$A_\nu U_1^{d_{1 \nu}} \cdots U_n^{d_{n \nu}}$ with $A_\nu \in {\mathbb K}[\X]$, $d_{i \nu} \ge 0$.
We associate with $P$ the following union of linear subspaces:
$$W_P:= \bigcup_{\nu, \mu} \{ d_{1 \nu}X_1 + \cdots + d_{n \nu}X_v= d_{1 \mu}X_1 + \cdots + d_{n \mu}X_n \},$$
where the union is taken over all pairs of different monomials.
(If there is at most one monomial with respect to $\U$, then $W_P$ is undefined.)

The following lemma is a version of the Lindemann-Weierstrass theorem.

\begin{lemma}\label{le:rational}
If for some point $\x=(x_1, \ldots ,x_n) \in \Real_{alg}^n$ the polynomial $Q:= P(\x, \U) \in \Real_{alg}[\U]$
is not identically zero and $Q(e^{x_1}, \ldots , e^{x_n})=0$, then $\x \in W_P$.
\end{lemma}

\begin{proof}
This is a slight adjustment of a standard proof of the Lindemann-Weierstrass theorem.

We have:
$$Q(e^{x_1}, \ldots ,e^{x_n})= \sum_\nu A_\nu(\x) e^{d_{1 \nu} x_1+ \cdots +d_{n \nu}x_n}=0,$$
where the coefficients $A_\nu(\x)$ are not all zero.
Removing all terms with zero coefficients, assume that in this sum all coefficients are non-zero.
Obviously, at least two terms will remain, one of which may be a non-zero constant.
By Baker's reformulation of the Lindemann-Weierstrass theorem \cite[Theorem~1.4]{Baker}, the powers
$d_{1 \nu} x_1+ \cdots +d_{n \nu}x_n$ are not pair-wise distinct.
It follows that $\x \in W_P$.
\end{proof}

Denote $f:= E(P)$, $V:=\ZZ(f) \subset \Real^n$, $V':=V \setminus W_P$.

\begin{lemma}\label{le:algebr}
$V$ is an {\em algebraic set} if and only if $V' \times \Real^n \subset \ZZ(P)$.
\end{lemma}

\begin{proof}
The ``if'' part of the implication is trivial.

Suppose now that $V$ is algebraic and $V' \neq \emptyset$.
Observe that the set $\{ \x \in V'|\> \{ \x \} \times \Real^n \subset \ZZ(P) \}$ is closed in $V'$
(with respect to the Euclidean topology).
Hence, the complement $V''$ of this set in $V'$ is open in $V'$.
Suppose that contrary to the claim, $V' \times \Real^n \not\subset \ZZ(P)$, i.e., $V'' \neq \emptyset$.
The algebraic points in $V$ are everywhere dense in $V$ since, by the assumption, $V$ is an algebraic set.
Therefore, there is a point $\x=(x_1, \ldots ,x_n) \in V'' \cap \Real_{alg}^n$ such that the polynomial
$Q:=P(\x, \U) \in \Real_{alg}[\U]$ is not identically zero.
Since $(e^{x_1}, \ldots , e^{x_n}) \in \ZZ(Q)$, we conclude, by Lemma~\ref{le:rational}, that $\x \in W_P$.
This contradicts the choice of $\x$.
\end{proof}

\begin{corollary}\label{cor:algebr1}
Let $V$ be an irreducible algebraic set with $\dim (V)=n-1$.
Then either $V$ is a hyperplane in $W_P$, or $V \times \Real^n$ is an irreducible component of $\ZZ(P)$.
\end{corollary}

\begin{proof}
If $V$ contains a hyperplane in $W_P$, then it coincides with this hyperplane, since $V$ is an algebraic set.
If $V$ is not a hyperplane in $W_P$, then $\dim (V \cap W_P)< n-1$.
Hence $V'_{n-1} \neq \emptyset$.
Since $V \times \Real^n$ is an irreducible algebraic set while, by Lemma~\ref{le:algebr},
$\dim (V \times \Real^n \cap \ZZ(P))=2n-1$, we conclude that $V \times \Real^n$ is an irreducible component of $\ZZ(P)$.
\end{proof}
Consider a polynomial $S\in {\mathbb K}[\X, \U]$.
Let $g:=E(S)$, $T:= \ZZ (g)$.
Suppose that $\dim (T)=n-1$, and that $T_{n-1} \subset B$, where $B$ is an $(n-1)$-dimensional
algebraic set defined over ${\mathbb K}$.
Represent $S$ as a polynomial in $\U$ with coefficients in ${\mathbb K} [\X]$.
Then every monomial is of the kind $A_\nu U_{1}^{d_{1 \nu}} \cdots U_{n}^{d_{n \nu}}$,
with $A_\nu \in {\mathbb K} [\X]$, $ d_{i \nu} \ge 0$.
Consider the algebraic set
$$W_S:= \bigcup_{\nu, \mu} \{ d_{1 \nu}X_{1}+ \cdots + d_{n \nu}X_{n} =
d_{1 \mu}X_{1}+ \cdots + d_{n \mu}X_{n} \},$$
where the union is taken over all pairs of different monomials.

\begin{lemma}\label{le:algebr1}
With the notations described above, the following inclusions take place:
\begin{enumerate}[(i)]
\item
$(T \setminus W_S)_{n-1} \times \Real^n \subset \ZZ(S)$;
\item
$(T \setminus (W_P \cup W_S))_{n-1} \times \Real^n \subset \ZZ(P,S)$;
\item
$(T \cap W_S) \setminus W_P)_{n-1} \times \Real^n \subset \ZZ(P)$.
\end{enumerate}
\end{lemma}

\begin{proof}
We will only prove item (i), since the proofs of the other items are essentially the same.

Denote $A:=(T \setminus W_S)_{n-1}$.
Suppose that $A \neq \emptyset$.
Observe that the set $\{ \x \in A|\> \{ \x \} \times \Real^n \subset \ZZ(S) \}$ is closed in $A$.
Thus, the complement $C$ of this set in $A$ is open in $A$.
Suppose that contrary to the claim, $A \times \Real^n \not\subset \ZZ(S)$, i.e., $C \neq \emptyset$.
The algebraic points in $T_{n-1}$ are everywhere dense in $T_{n-1}$ since $T_{n-1} \subset B$.
Therefore, there is a point $\x=(x_1, \ldots ,x_n) \in C \cap \Real_{alg}^n$ such that polynomials
$$Q:=S(\x, U_1, \ldots , U_n) \in \Real_{alg}[\U]$$
are not identically zero.
Since $(e^{x_1}, \ldots , e^{x_n}) \in \ZZ(Q)$, we conclude, by Lemma~\ref{le:rational}, that $\x \in W_S$.
This contradicts the choice of $\x$.
\end{proof}

Now we assume that $V=\ZZ(f) \subset \Real^n$ and $T=\ZZ(g) \subset \Real^n$ are exponential sets,
not necessarily algebraic.
We will associate with these sets polynomials $P,\ S$ and sets $W_P,\ W_S$ as above.

\begin{lemma}\label{le:structure}
Let $\dim (V)= n-1$ and $\ZZ(P)$ be irreducible.
Let $T \subset V$ be a $(n-1)$-dimensional irreducible component of $V$, with irreducible $\ZZ(S)$,
such that there is an $(n-1)$-dimensional algebraic set $B$ (defined over $\mathbb K$) containing $T_{n-1}$.
Then either $V= T$ or $T_{n-1} \subset W_P$.
\end{lemma}

\begin{proof}
Let $A:= (T \setminus (W_P \cup W_S))_{n-1}$.
Suppose first that $A \neq \emptyset$.
By Lemma~\ref{le:algebr1} (ii), $A \times \Real^n \subset \ZZ(S)$ and $A \times \Real^n \subset \ZZ(P)$.
Because $\dim (A \times \Real^n)=2n-1$, and the sets $\ZZ(S),\> \ZZ(P)$ are irreducible algebraic, these sets coincide.
It follows that $T= V$.

Now suppose that $A = \emptyset$ and $\dim ((T \cap W_S) \setminus W_P)=n-1$.

By analytic continuation, it means that $(T \cap W_S)_{n-1}$ consists of some hyperplanes
in $W_S$ which are not all in $W_P$, hence $(T \cap W_S) \setminus W_P$ contains a hyperplane,
say $L$, in $\Real^n$.
By Lemma~\ref{le:algebr1} (iii), $(T \cap W_S) \setminus W_P)_{n-1} \times \Real^n \subset \ZZ(P)$.
Therefore, $\ZZ(P)$ contains a hyperplane $L \times \Real^n$, hence $\ZZ(P)$ (being irreducible algebraic set)
coincides with $L \times \Real^n$.
It follows that both $T$ and $V$ coincide with the same hyperplane, $L$, in $\Real^n$,
thus again, $T= V$.

If neither of the above alternatives take place, we have $T_{n-1} \subset W_P$.
\end{proof}

\begin{remark}
Lemma~\ref{le:structure} can be viewed as an analogy, in codimension 1, of the classical Ax-Lindemann theorem
(see its modern treatment in \cite[Section~6]{Pila}) which deals with exponential
sums over complex numbers.
\end{remark}

\begin{corollary}\label{cor:algebr}
Let $V \subset \Real^n$ be an $(n-1)$-dimensional algebraic set over $\mathbb K$, which is irreducible as an algebraic set.
Then it's irreducible.
\end{corollary}

\begin{proof}
Let $T$ be an $(n-1)$-dimensional irreducible exponential component of $V$.
Then by Lemma~\ref{le:structure}, either $T= V$, or $T_{n-1} \subset W_P$.
In the former case we are done.
In the latter case, by analytic continuation, $T_{n-1}$ contains a hyperplane.
Therefore, $V$ also contains this hyperplane, moreover, being an irreducible algebraic set, coincides with this hyperplane.
It follows that $T= V$.
\end{proof}

\section{Proof of the main theorem}\label{sec:main_theorem}

Schanuel's conjecture over real numbers is the following statement.
\bigskip

{\em Suppose that for real numbers $x_1, \ldots , x_n$ the transcendence degree
$${\rm td}_{\mathbb Q}(x_1, \ldots ,x_n,e^{x_1}, \ldots ,e^{x_n}) <n.$$
Then there are integers $m_1, \ldots ,m_n$, not all zero, such that $m_1x_1 + \cdots + m_nx_n=0$.}
\bigskip

This statement (along with its other versions) is the central, yet unsettled, conjecture in transcendental number theory
(see \cite{Lang, Kirby}).

Throughout this section we will assume that for $P \in {\mathbb K}[\X,\U]$ the real algebraic set
 $\ZZ(P)\subset \Real^{2n}$ is irreducible, and that for $f=E(P)$
 the exponential set $V:=\ZZ(f) \subset \Real^{n}$, is a hypersurface, i.e., $\dim (V)=n-1$.
 The case $n=1$ is covered in Section~\ref{sec:one_exp}, so in the sequel assume that $n>1$.

\begin{lemma}\label{le:alternative}
Assuming Schanuel's conjecture, every $(n-1)$-dimensional irreducible component of $V$ either coincides with $V$
($V$ is irreducible), or it is the finite union of hyperplanes through the origin, defined over $\Q$,
and a set of points having local dimension less than $n-1$.
\end{lemma}

\begin{proof}
Suppose $V$ is reducible and $T$ is its irreducible component having dimension $n-1$.
Then $T= \ZZ(g) \subset \Real^n$ for a suitable $E$-polynomial $g$ such that $g=E(S)$, where $S \in {\mathbb K}[\X, \U]$.

Let $\dim \ZZ(P)=m$ for some $n-1 \le m \le 2n-1$, and $\dim (\ZZ(P, S))= \ell$ for some $n-1 \le \ell$.
Observe that $\ell <m$, otherwise $\ZZ(P) \subset \ZZ(S)$ since $\ZZ(P)$ is irreducible, which contradicts
the existence of components of $V$ different from $T$.
In particular, $n \le m$ and $\ell \le 2n-2$.

The projection of the $(n-1)$-dimensional set
$$\ZZ(P, S, U_1-e^{X_1},\ldots , U_n-e^{X_n})=\ZZ(S, U_1-e^{X_1}, \ldots , U_n-e^{X_n})$$
to a coordinate subspace of some $n-1$ coordinates $X_1, \ldots ,X_{\alpha-1},X_{\alpha+1}, \ldots , X_n$,
where $1 \le \alpha \le n$, is $(n-1)$-dimensional.
Consider any such $\alpha$.
Then the projection contains a dense (in this projection) set of points
$(x_1, \ldots ,x_{\alpha -1},x_{\alpha +1}, \ldots ,x_n) \in \Real_{alg}^{n-1}$
and for each such point the intersection
$$\ZZ(S, P, X_1-x_1, \ldots ,X_{\alpha -1}-x_{\alpha -1}, X_{\alpha +1}-x_{\alpha +1}, \ldots , X_n-x_n )$$
is an {\em algebraic set defined over $\mathbb K$}.

Observe that the set of points $(x_1, \ldots ,x_{\alpha -1},x_{\alpha +1}, \ldots ,x_n)$ such that the dimension
$$\dim (\ZZ( S, P, X_1-x_1, \ldots ,X_{\alpha -1}-x_{\alpha -1}, X_{\alpha +1}-x_{\alpha +1}, \ldots , X_n-x_n ))$$
is larger than $\ell -n+1$ is a semialgebraic set in $\Real^{n-1}$ having dimension less than $n-1$.
Hence, for a dense subset of algebraic points $(x_1, \ldots ,x_{\alpha -1},x_{\alpha +1}, \ldots ,x_n)$ in $\Real^{n-1}$
the dimension of the algebraic set
$$\ZZ(S, P, X_1-x_1, \ldots ,X_{\alpha -1}-x_{\alpha -1}, X_{\alpha +1}-x_{\alpha +1}, \ldots , X_n-x_n )$$
is at most $\ell -n+1$, i.e., at most $n-1$.

Represent $P$ as a polynomial in $\U$ with coefficients in ${\mathbb K} [\X]$.
Every monomial is then of the kind
$$A_\nu U_{1}^{d_{1 \nu}} \cdots U_{n}^{d_{n \nu}},$$
with $A_\nu \in {\mathbb K}[\X]$, $d_{j \nu} \ge 0$.
Consider the real algebraic set
$$
B:= \bigcup_\nu \{A_\nu=0 \} \cup
\bigcup_{\nu,\mu} \{ d_{1 \nu}X_{1}+ \cdots + d_{n \nu}X_{n} = d_{1 \mu}X_{1}+ \cdots + d_{n \mu}X_{n} \},
$$
where the first union is taken over all monomials, while the second union is taken over all pairs of different monomials.

Suppose first that $\dim (T \setminus B)< n-1$. Then $T_{n-1} \subset B$.
By Lemma~\ref{le:structure}, either $V= T$ or $T_{n-1} \subset W_P$.
The first of these alternatives contradicts the reducibility of $V$, hence, $T_{n-1}$ is a union of rational hyperplanes
through the origin, and the lemma is proved.

Suppose now that $\dim (T \setminus B)= n-1$.
Then there exists a number $\alpha,\ 1 \le \alpha \le n$, a point
$(x_1, \ldots ,x_{\alpha -1}, x_{\alpha +1}, \ldots ,x_n) \in \Real_{alg}^{n-1}$,
and a number $x_\alpha \in \Real$ such that
\begin{enumerate}
\item
$(x_1, \ldots ,x_n) \in T \setminus B$;
\item
the numbers $x_j$, where $j \in \{1, \ldots, \alpha -1,\alpha +1, \ldots ,n \}$, are linearly independent over $\Q$;
\item
the dimension of
$$\ZZ( S, P, X_1-x_1, \ldots ,X_{\alpha -1}-x_{\alpha -1}, X_{\alpha +1}-x_{\alpha +1}, \ldots , X_n-x_n )$$
is at most $n-1$.
\end{enumerate}
Let $(x_1, \ldots ,x_{\alpha -1}, x_{\alpha +1}, \ldots ,x_n)$ be such point.
Then the set
$$\ZZ( S, U_1-e^{X_1}, \ldots , U_n-e^{X_n} ) \subset \Real^{2n}$$
contains a point, namely, $(x_1, \ldots , x_n, e^{x_1}, \ldots , e^{x_n})$, which also lies in an algebraic set
$$\ZZ( S,\ P, X_1-x_1, \ldots ,X_{\alpha -1}-x_{\alpha -1}, X_{\alpha +1}-x_{\alpha +1}, \ldots , X_n-x_n )$$
of dimension at most $n-1$.
By Schanuel's conjecture, $m_1x_1+ \cdots +  m_nx_n=0$ for some integers $m_1, \ldots , m_n$, not all equal to 0.
Since $(x_1, \ldots ,x_{\alpha -1}, x_{\alpha +1}, \ldots ,x_n)$ are algebraic numbers, linearly independent over $\Q$,
we have $m_\alpha \neq 0$, hence, $x_\alpha$ is also algebraic.

Thus, the point $(x_1, \ldots ,x_n)$ has real algebraic coordinates.
Then $(x_1, \ldots ,x_n) \in B$, either because all coefficients $A_\nu$ vanish (hence the polynomial
$P(x_1, \ldots ,x_n, \U)$ is identically zero with respect to $\U$), or otherwise,
by Lemma~\ref{le:rational}, since
$$(e^{x_1}, \ldots , e^{x_n}) \in \ZZ(S(x_1, \ldots ,x_n, \U)).$$
This contradicts condition (1).
It follows that components $T$, with the property
$$\dim (T \setminus B)= n-1,$$
do not exist.
\end{proof}

\begin{remark}
In the proof of Lemma~\ref{le:alternative} the following implication of Schanuel's conjecture was actually used,
rather than Schanuel's conjecture {\em per se}.
If numbers
$$x_1, \ldots ,x_{\alpha -1}, x_{\alpha +1}, \ldots ,x_n \in \Real_{alg}$$
are linearly independent over $\Q$, $x_\alpha \in \Real$, and the transcendence degree of
$$x_1, \ldots ,x_n, e^{x_1}, \ldots , e^{x_n}$$
is less than $n$, then $x_\alpha \in \Real_{alg}$.
It is not known whether this particular case of Schanuel's conjecture is true.
As M. Waldschmidt pointed out \cite{waldpriv}, this particular case implies, for $n=2$,
that $e$ and $\log 2$ are algebraically independent, which is not known.
\end{remark}

\begin{proof}[Proof of Theorem~\ref{thm:main}]
Suppose that $V:=\ZZ(f)$ is reducible and $T:=\ZZ(g)$ is its irreducible component.
Then, according to Lemma~\ref{le:alternative}, $T$ is the  union of the set $T^{(1)}$ of rational hyperplanes through zero,
and a set $T^{(2)}$ of points of some local dimensions less than $n-1$.
Suppose that $T^{(2)} \neq \emptyset$, and the maximum of these dimensions is $p <n-1$.
According to Theorem~\ref{th:subset}, there is an exponential set $T^{(3)}$ such that $T^{(2)} \subset T^{(3)} \subset T$
and $\dim T^{(3)}=p$.
If follows that $T=T^{(1)} \cup T^{(3)}$, hence, $T$ is reducible which is a contradiction.
Therefore, $T^{(2)} =\emptyset$ and $T=T^{(1)}$.
Since $T$ is irreducible, the set $T^{(1)}$ consists of a unique hyperplane.
\end{proof}

Recall the definition of the ring ${\mathbb K}[\X, e^{\X_k}]$, for $0 \le k \le n$, in Remark~\ref{re:less_exp}.
Assuming Schanuel's conjecture, the following statement is a generalization of Corollary~\ref{cor:algebr}.

\begin{corollary}
Let $P \in {\mathbb K}[\X, \U_k]$ and assume that $\ZZ (P) \in \Real^{n+k}$ is irreducible.
Let $f:=E(P) \in {\mathbb K}[\X, e^{\X_k}]$ and $\ZZ (f)$ be an $(n-1)$-dimensional exponential set, irreducible in
${\mathbb K}[\X, e^{\X_k}]$.
Then, assuming Schanuel's conjecture, $\ZZ (f)$ is irreducible.
\end{corollary}

\begin{proof}
Let $\ZZ (f)$ be reducible.
By Theorem~\ref{thm:main}, all $(n-1)$-dimensional irreducible components of $\ZZ (f)$ are rational hyperplanes,
while by Theorem~\ref{th:subset} and Remark~\ref{re:less_exp}, the union of all the rest of irreducible components
is an exponential set defined in ${\mathbb K}[\X, e^{\X_k}]$.
Since any rational hyperplane is an irreducible set defined in ${\mathbb K}[\X, e^{\X_k}]$, we conclude that
$\ZZ (f)$ is also reducible in ${\mathbb K}[\X, e^{\X_k}]$.
\end{proof}

\section{Open questions}

1.\quad 
Let $\dim (\ZZ (f))=m \le n-1$, the algebraic set $\ZZ (P) \subset \Real^{2n}$ be irreducible, and $\dim (\ZZ (P))=k$.
Obviously, $m \le k$.
Assume that $k \le m+n$.
(In the case of a single exponential, in Section~\ref{sec:one_exp}, the condition $k \le m+1$ was achieved by
introducing an admissible set.)
A weak conjecture is that, under Schanuel's conjecture, either $\ZZ (f)$ is irreducible or every $m$-dimensional
irreducible component of $\ZZ (f)$ is contained in a rational hyperplane through the origin.
\medskip

2.\quad 
We conjecture that there is an upper bound $(cd)^n$ on the number of all irreducible components of $\ZZ (f)$,
similar to the bound in Corollary~\ref{cor:bezout}.
Here $c$ is an absolute constant and $d= \deg (P)$.
\medskip

3.\quad 
It would be interesting to understand the structure of absolutely irreducible components of $\ZZ (f)$
(i.e., irreducible over $\Real$).
Example~\ref{ex:irred} shows that Theorem~\ref{thm:main} is no more true in this case.
\medskip

4.\quad Let $P \in {\mathbb K}[\X, U_1]$, $\z \in \ZZ (P, U_1-e^{X_1}) \setminus \ZZ (X_1)$, and
${\mathcal U}$ be a neighbourhood of $\z$ in $\Real^{n+1}$.
Set $\ZZ (P) \cap {\mathcal U}$ admits a Whitney stratification, while $\ZZ (U_1 -e^{x_1})$ is
a real analytic submanifold of $\Real^{n+1}$.

In Appendix, Theorem~\ref{th:transverse}, it is proved that the intersection of Whitney stratified sets,
$\ZZ (P) \cap {\mathcal U}$ and $\ZZ (U_1- e^{X_1}) \cap {\mathcal U}$, is transverse.
We conjecture that transversality remains true in the general case of many exponentials.

\section{Appendix}

In this section we prove a transversality property for $E$-polynomials depending on a single exponential.

The following statement is well known to experts but we could not find an exact reference to it in literature.

\begin{proposition}\label{prop:alg_strata}
Let ${\mathcal X} \subset \Real^n$ be an intersection of an algebraic set and an open set.
Then there is a Whitney stratification of $\mathcal X$ (with connected strata) such that for each stratum
$S$ there is an open set $\mathcal U$ containing $S$ such that $S$ coincides with the intersection of
$\mathcal U$ with an algebraic set.
Moreover, if the algebraic set in $\mathcal X$ is defined over $\Real_{alg}$ then the algebraic set in $S$
is defined over $\Real_{alg}$.
\end{proposition}

\begin{proof}
Let $\widehat {\mathcal X} \subset \C^n$ be the complexification of $\mathcal X$
(i.e., the Zariski closure of $\mathcal X$ in $\C^n$).
Teissier's theorem \cite[Ch. VI, Proposition 3.1]{Teiss}
implies that $\widehat {\mathcal X}$ admits a Whitney stratification,
with each stratum being a Zariski locally closed set in $\C^n$.
By taking connected components of real parts of strata, this stratification induces the required
Whitney stratification on $\mathcal X$.

To prove the second statement of the proposition, observe that, according to \cite{Rannou}, the existence of
the required Whitney stratification for a fixed $\mathcal X$ can be expressed by a formula of the
first-order theory of real closed fields.
Now the statement follows from the {\em transfer principle} in real closed fields \cite[Proposition~5.2.3]{BCR}.
\end{proof}

Let $P \in {\mathbb K}[\X, U_1]$, $f=E(P)$, $V= \ZZ (f)$, and $\pi:\> \Real^{n+1} \to \Real^{n}$ be
the projection map along $U_1$.

\begin{theorem}\label{th:transverse}
Assume that for some $0 \le p \le \dim (V)$ there is a point $\x=(x_1, \ldots x_n) \in V_p$ with $x_1\neq 0$.
Let $\z \in \ZZ(P, U_1-e^{X_1} )$ be such that $\pi (\z)= \x$.
Then we have $\dim_{\z} (\ZZ(P))=p+1$.
\end{theorem}

\begin{proof}
Let $\mathcal U$ be a neighbourhood  of $\z$ in $\Real^{n+1}$.
By Proposition~\ref{prop:alg_strata} there exists a  Whitney stratification of  $\ZZ(P) \cap {\mathcal U}$.
Let $S$ be a stratum of this stratification containing $\z$, and let $S$ be an intersection of an
algebraic set $S'$ and an open set in $\Real^{n+1}$.

Note that $S$ and $\ZZ(U_1-e^{X_1})$ are real analytic submanifolds of $\Real^{n+1}$.

To begin, we prove the following claim.

{\bf Claim:}\quad The manifolds $S$ and $\ZZ( U_1-e^{X_1})$ cannot be tangent at $\z$.

This claim implies that if $\dim (S)>0$, then $S$ and $\ZZ( U_1-e^{X_1})$ are transverse at $\z$ in $\Real^{n+1}$.

To verify this  claim we proceed by  induction on $\dim S$.
Since any algebraic point in $\ZZ(U_1-e^{X_1})$ will require $x_1=0$, the base case of the induction,
with $\dim (S)=0$, is immediate.

For the induction step  assume that  $\dim (S)=n-k+1$ for some $1 \le k < n+1$.
Then, we can deduce from Lemma \ref{le:3.2.9} the existence of a neighbourhood $\mathcal V$ of $\z$ in $\Real^{n+1}$ such that
$$\mathcal V \cap S= \mathcal V \cap \ZZ(P_1,\cdots,P_k),$$ where
all $P_i$ are polynomials in $I(S')$, and the Jacobian $(k \times (n+1))$-matrix of the system
$P_1= \cdots =P_k=0$ has the maximal rank $k$ at $\z$.
Now, assume that $S$ and $\ZZ(U_1-e^{X_1})$  intersect tangentially at $\z$.
Then all $(k+1) \times (k+1)$-minors of the Jacobian $(k+1) \times (n+1)$-matrix
$$\frac{ \partial (U-e^{X_1}, P_1, \ldots , P_k)}{\partial(U,X_1, \ldots , X_n)}$$
vanish at $\z$.
In particular, all of the following minors vanish:
\begin{equation}\label{eq:jacob1}
 \frac{\partial (U-e^{X_1}, P_1, \ldots , P_k)}{\partial (U, X_1, X_{i_1}, \ldots ,X_{i_{k-1}})}
\end{equation}
for all subsets $\{ i_1, \ldots ,i_{k-1} \} \subset \{2, \ldots ,n \}$.

Also, if $k <n$, all of the following minors vanish:
\begin{equation}\label{eq:jacob2}
\frac{\partial(P_1, \ldots ,P_k)}{\partial(X_{i_1}, \ldots ,X_{i_k})}
\end{equation}
for all subsets $\{ i_1, \ldots ,i_k \} \subset \{2, \ldots ,n \}$.

Clearly, the  determinant of   (\ref{eq:jacob1}) equals
$$D (U,X_1, \ldots ,X_n)= \det \frac{\partial (P_1, \ldots, P_k)}{\partial (X_1,X_{i_1}, \ldots ,X_{i_{k-1}})}
- e^{X_1} \det \frac{ \partial (P_1, \ldots ,P_k)}{\partial (U, X_{i_1}, \ldots , X_{i_{k-1}})}.$$
Define $\widehat{D}$ by replacing $e^{X_1}$ by $U$ in $D$.
Then $\widehat D(\z)= D(\z)$.

Observe that
$$
A(U, X_1, \ldots ,X_n):=\det \frac{\partial (P_1, \ldots ,P_k)}{\partial (U,X_{i_1}, \ldots ,X_{i_{k-1}})} \neq 0
$$
at $\z$ for some subset $\{ i_1, \ldots ,i_{k-1} \} \subset \{2, \ldots ,n \}$.
Indeed, otherwise for all subsets $\{ i_1, \ldots ,i_{k-1} \}$ the condition $D(U,X_1, \ldots ,X_n)=0$ would imply that
$$B(U, X_1, \ldots ,X_n):=\det \frac{\partial (P_1, \ldots, P_k)}{\partial (X_1,X_{i_1}, \ldots ,X_{i_{k-1}})} =0$$
at $\z$.
Hence, all $k \times k$-minors for the system $P_1= \cdots =P_k=0$ vanish at $\z$, taking into the account that
all minors (\ref{eq:jacob2}) vanish at $\z$ when $k<n$.
This contradicts the supposition that the Jacobian matrix of the system has the maximal rank at $\z$.

We conclude that $A(U, X_1, \ldots ,X_n) \neq 0$ at $\z$ for some subset $\{ i_1, \ldots ,i_{k-1} \} \subset \{2, \ldots ,n \}$.
Fix such a subset $\{ i_1, \ldots ,i_{k-1} \} \subset \{2, \ldots ,n \}$.
Then we can consider $P_1= \cdots =P_k=0$ as an implicit map $F=(F_1,F_{i_1}, \ldots , F_{i_{k-1}})$
from the  vector space of variables
$X_1, X_{j_1}, \ldots ,X_{j_{n-k}}$ to the vector space of variables $U, X_{i_1}, \ldots ,X_{i_{k-1}}$, with
$$\{ j_1, \ldots ,j_{n-k} \}= \{2, \ldots , n\} \setminus \{i_1, \ldots , i_{k-1} \}.$$
In particular, there is a differentiable function $F_1(X_1, X_{j_1}, \ldots ,X_{j_{n-k}})=U$, whose partial derivative
with respect to $X_1$ in the neighbourhood of $\z$ is given, according to formulae for differentiating of implicit functions, by
$$\frac{\partial F_1}{\partial X_1}(X_1, X_{j_1}, \ldots ,X_{j_{n-k}})= - \frac{B(U, X_1, \ldots ,X_n)}
{A(U, X_1, \ldots ,X_n)}.$$

Suppose that $\widehat D$ vanishes identically in the neighbourhood of $\z$ in $S$.
Then, in the neighbourhood,
$$U= \frac{B(U, X_1, \ldots ,X_n)}{A(U, X_1, \ldots ,X_n)},$$
and therefore,
$$\frac{\partial F_1}{\partial X_1}(X_1, X_{j_1}, \ldots ,X_{j_{n-k}})=- F_1(X_1, X_{j_1}, \ldots ,X_{j_{n-k}}).$$
Let $G$ be the restriction of $F_1$ to the straight line $\ZZ( X_{j_1}-x_{j_1}, \ldots , X_{j_{n-k}}-x_{j_{n-k}})$.
Then $G$ satisfies the differential equation $dG/dX_1=-G$, hence $G(X_1)=e^{-X_1}$.
Since $\ZZ(G(X_1)-U )$ is a semialgebraic curve at $\z$, we get a contradiction.
Therefore, $\widehat D$ does not vanish identically in the neighbourhood of $\z$ in $S$.

It follows that $\dim_{\z} (\widehat D \cap S) <n-k+1$.
The set $\widehat D \cap S$ is either smooth at $\z$, or $\z$ is its singular point.
In the first case, $T_{\z}(\widehat D \cap S) \subset T_{\z}(S)$, hence $\widehat D \cap S$ is
tangent to $\ZZ(U-e^{X_1})$ at $\z$, which is impossible by the inductive hypothesis.
In the second case, $\z$ belongs to a stratum of a smooth stratification of $\widehat D \cap S$, which has even
smaller dimension than $\dim_{\z} (\widehat D \cap S)$.
This is again impossible by the inductive hypothesis. Therefore, $S$ and $\ZZ( U-e^{X_1} )$
do not meet tangentially at $\z$.
The claim is proved.

To finish the proof of the theorem, we can assume that $S$ is transverse to $\ZZ(U-e^{X_1})$ at $\z$.
Let $R$ be any other stratum of the stratification such that $S \subset \overline R$.
Since $\ZZ (U-e^{X_1} )$ is an oriented hypersurface in $\Real^n$, there are two points ${\bf a}, {\bf b} \in S$
on different sides of $\ZZ(U-e^{X_1})$.
There is an open curve interval $\gamma \subset R$ such that ${\bf a}, {\bf b} \in \overline \gamma$.
Then $\gamma \cap \ZZ(U-e^{X_1}) \neq \emptyset$, thus $\ZZ( U-e^{X_1} ) \cap R \neq \emptyset$.
Since, by \cite{Tr}, Whitney's $(a)$-regularity implies Thom's $(t)$-regularity, the manifolds $\ZZ(U-e^{X_1})$ and $R$ are
transverse in a neighbourhood of $\z$.
But $\dim_{\z} (\ZZ(P, U-e^{X_1}))=p$, while $\dim (\ZZ( U-e^{X_1}))=n-1$.
It follows that $\dim_{\z} (\ZZ(P))=p+1$.
\end{proof}

\section*{Acknowledgements}
We would like to  thank L. Birbrair, N. Dutertre, and D. Trotman for communicating to us a proof of the first part of
Proposition \ref{prop:alg_strata}.
Part of the research presented in this article was carried out during visits of the second author to the Aalto Science Institute in 2015/16.
We are very grateful for financial support for these visits from the  Aalto Visiting Fellows programme.

\end{document}